\documentclass[a4paper,10pt]{article}
\usepackage[T1]{fontenc}
\usepackage{amsthm,amssymb,amsmath}
\newtheorem{theorem}{Theorem}
\newtheorem{lemma}[theorem]{Lemma}
\newtheorem{corollary}[theorem]{Corollary}

\title{Another approach to non-repetitive colorings of graphs of bounded degree}
\author{Matthieu Rosenfeld\footnote{CNRS, LIS, Aix Marseille Universit\'e, Universit\'e de Toulon, Marseille, France}}
\date{}
\begin{document}
\maketitle
\begin{abstract}
 We propose a new proof technique that aims to be applied to the same problems as the  Lov\'asz Local Lemma or the entropy-compression method.
 We present this approach in the context of non-repetitive colorings and we use it to improve upper-bounds relating different non-repetitive numbers to the maximal degree of a graph.
It seems that there should be other interesting applications to the presented approach.
 
 In terms of upper-bound our approach seems to be  as strong as entropy-compression, but the proofs are more elementary and shorter.
 The application we provide in this paper are upper bounds for graphs of maximal degree at most $\Delta$: a minor improvement on the upper-bound of the non-repetitive number,
 a $4.25\Delta +o(\Delta)$ upper-bound on the weak total non-repetitive number and
 a $ \Delta^2+\frac{3}{2^\frac{1}{3}}\Delta^{\frac{5}{3}}+ o(\Delta^{\frac{5}{3}})$ upper-bound on the total non-repetitive number of graphs.
 This last result implies the same upper-bound for the non-repetitive index of graphs, which improves the best known bound.

\end{abstract}

\section{Introduction}
A sequence $s_1\ldots s_{2n}$ is \emph{a square} if
$s_i=s_{i+n}$  for each $i\in\{1,\ldots,n\}$.
A sequence is \emph{repetitive} if it contains a consecutive subsequence that is a square 
and it is \emph{non-repetitive} (or \emph{square-free}) otherwise. 
For instance, the words \textbf{hotshots}, \textbf{repetitive} and \textbf{alfalfa} are repetitive and
the words \textbf{total} and \textbf{minimize} are square-free.

The work of Thue on non-repetitive words is regarded as the starting point of combinatorics on words
\cite{Thue06,Thue1} (see \cite{BerstelThue} for a translation in modern mathematical English).  He showed that there are infinite sequences over three elements that are square free.
Many generalizations and variations of this notion have been studied.
One such notion  that received a lot of attention is the notion of 
non-repetitive coloring of graphs introduced by Currie and popularized by an article by Alon et Al. \cite{alongraph,curriegraph} 
(see \cite{CZERWINSKI2007453,vSkrabuvlakova2015TheTC} for surveys on this topic).
We say that a coloring (either of the vertices or of the edges) of a graph is \emph{non-repetitive} if the sequence of colors induced by any path is non-repetitive.
The \emph{Thue number} (resp. \emph{Thue index})(also called  \emph{non-repetitive number} and \emph{non-repetitive index}) of a graph, denoted by $\pi(G)$ (resp. $\pi'(G)$) is the smallest number of colors in a non-repetitive coloring of the vertices (resp. the edges) of the graph. 
Alon et Al. showed that $\pi'(G)$ is in $O(\Delta^2)$ where $\Delta$ is the maximum degree of $G$ \cite{alongraph}.  
Different authors provided successive improvements of these bounds both for the Thue number and the Thue index \cite{HARANT2012374, Dujmovic2016, MontassierEntropie}.
Although this is not the topic of this article, non-repetitive colorings have since then been studied in many other context than graphs of bounded degree. For instance, after many intermediate results it was recently showed that planar graphs have bounded non-repetitive number \cite{planargraphs}. 

Most results regarding non-repetitive coloring of graphs of bounded maximal degree are either based on the Lov\'asz Local Lemma or entropy compression and they naturally hold on the stronger setting of
list coloring. A \emph{list assignment} of a graph $G$ is a function that maps any vertex $v$ (and/or any edge) to a set of colors $L(v)$.
A graph $G$ is \emph{non-repetitively $L$ colorable} if there is a way to non-repetitively color it such that the color of each vertex $v$ belongs to its list of colors $L(v)$. A graph is then said to be \emph{non-repetitively $l$-choosable} if it is  non-repetitively $L$ colorable for any list assignment $L$ such that every list is of size at least $l$.
The non-repetitive choice number $\pi_{\operatorname{ch}}(G)$ is the smallest integer $l$ such that $G$ is non-repetitively $l$-choosable. Similarly, the list variant of the non-repetitive index is denoted by $\pi'_{\operatorname{ch}}(G)$.
The best bounds relating these quantities to the maximum degree $\Delta$ of a graph are respectively (see \cite{MontassierEntropie})
\begin{equation}
 \pi'_{\operatorname{ch}}(G)\le\Delta^2+2^\frac{4}{3}\Delta^\frac{5}{3}+O(\Delta^\frac{4}{3})
\end{equation}
and 
\begin{equation}\label{ThueChoiceNumberToImprove}
 \pi_{\operatorname{ch}}(G)\le\left\lceil\Delta^2+\frac{3}{2^\frac{2}{3}}\Delta^\frac{5}{3}+\frac{2^\frac{2}{3}\Delta^\frac{5}{3}}{\Delta^\frac{1}{3}-2^\frac{1}{3}}\right\rceil\,.
\end{equation}

The notion of \emph{Total Thue coloring} was introduced by
Schreyer and  \v{S}krabuvl\'akov\'a \cite{introthuetotalcol}.
A coloring of the edges and the vertices of a graph is a \emph{weak total Thue coloring} if the sequence of consecutive vertex-colors and edge-colors of every path is non repetitive. If moreover the sequence of vertex-colors and the sequence of edge-colors of any path are both non-repetitive then this is a (strong) \emph{total Thue coloring}.
The weak total Thue number $\pi_{\operatorname{T_w}}(G)$ (resp. $\pi_{\operatorname{T}}(G)$) is the minimum number of colors in a weak total Thue coloring  of $G$ (resp. a total Thue coloring of $G$). These two parameters both have there list-coloring counterpart denoted respectively by  $\pi_{\operatorname{T_wch}}(G)$ and $\pi_{\operatorname{Tch}}(G)$.
In their article, Schreyer and  \v{S}krabuvl\'akov\'a showed that
$\pi_{\operatorname{T}}(G)\le 5\Delta^2+o(\Delta^2)$, 
$\pi_{\operatorname{Tch}}(G)\le 17.9856\Delta^2$ and 
$\pi_{\operatorname{T_w}}(G)\le |E(G)|-|V(G)|+5$ \cite{introthuetotalcol}.
We remark that the second bound also relies on an application of the Lov\'asz Local Lemma.

In this article, we propose a different proof technique strongly related to the Lov\'asz Local Lemma and to entropy compression (whose idea is based on the algorithmic proof of the Lov\'asz Local Lemma by Moser and Tardos \cite{10.1145/1667053.1667060})  and we apply this technique to different non-repetitive coloring problems.
This technique provides bounds as good as entropy compression, but is much more elementary. 
The more advanced piece of mathematics required in the proof is summation of geometric series (for comparison, entropy compression arguments usually rely on Analytic combinatorics to compute some variations of Catalan numbers to bound the number of records). 
The main idea of this approach is to show inductively that at every ``step'' of the coloring the number of possible colorings grows exponentially (this resembles the proof of LLL in this regard).
One more benefit of this approach is that it provides exponential lower bounds on the number of solutions. However, we lose the constructive aspect and the algorithmic implications of entropy-compression arguments.
We should mention that Bernshteyn recently introduced the Local Cut Lemma \cite{BERNSHTEYN201795}, a lemma that aims to be applied to the same set of problems as LLL or entropy compression, but is more powerfull.
In some cases it also provides shorter proofs than LLL or entropy-compression, but they are still more technical (and it might be argued that they are longer if one includes the proof of the Local Cut Lemma itself).

As a simple illustration of our technique we first provide a proof that the Thue choice number of any path is at most $4$ (for a proof of this result based on LLL see \cite{doi:10.1002/rsa.20347} and for a proof based on entropy-compression see \cite{10.1002/rsa.20411}).
Then we apply our technique to the Thue choice number in Theorem \ref{improvingthuenumber} and improve the lowest degree term of the bound given by \eqref{ThueChoiceNumberToImprove}. The  improvement is minor and could certainly be achieved with a more careful analysis in \cite{MontassierEntropie}, however the proof given in  \cite{MontassierEntropie} is much more technical than our self-contained proof of one and a half page.

We then apply our method to weak total Thue coloring and total Thue coloring. Our first result given in  Theorem \ref{weaktotalthuecoloringthm} asserts that $\pi_{\operatorname{T_wch}}(G)\le6\Delta$ for any graph of maximum degree $\Delta$. Prior to our article, there does not seems to be any other known result than the $\pi_{T_w}(G)\le |E(G)|-|V(G)|+5$ from \cite{introthuetotalcol}. These two results do not directly compare to each others, 
but our result is stronger as long as the average degree is at most $14$ and in many other reasonable graph classes (also our result is stronger since for any graph $G$,  $\pi_{\operatorname{T_w}}(G)\le\pi_{\operatorname{T_wch}}(G)$).
Remark that our result also implies a linear bound on the number of vertices instead of a quadratic bound.
Then, regarding the total Thue choice number, we show in Theorem \ref{weaktotalthuecoloringthm} that $\pi_{T}\le\pi_{\operatorname{Tch}}(G)\le\Delta^2+O(\Delta^\frac{5}{3})$ which improves considerably the two bounds from \cite{introthuetotalcol} previously mentioned.
We conclude with a discussion regarding applications of this method to other problems.

We assume that our reader is familiar with basic graph notations.
Our results are all about some form of list-coloring problem and we allow ourselves to write ``coloring of $G$'' instead of ``coloring of $G$ respecting the list assignment $L$'' whenever $L$ is obvious from the context.

\section{Non-repetitive colorings of paths and the proof technique} 
In this section, we first give an illustrative example of our proof technique. 
Then we informally sketch a more general description of the proof technique.

\subsection{Non-repetitive colorings of paths}
This section is devoted to the proof of Theorem \ref{thmpath}.
This result is not new \cite{doi:10.1002/rsa.20347,10.1002/rsa.20411}, but it is simple application of our approach.
\begin{theorem}\label{thmpath}
For every path $P$, $\pi_{\operatorname{ch}}(P)=4$.
\end{theorem}
This Theorem is a simple consequence of the following lemma.
We order the vertices of any path ``from left to right'' such that each vertex is adjacent to the vertex to its right and to its left (the leftmost and rightmost vertices are the ends of the path).
\begin{lemma}\label{lemmpath}
Let $L$ be a list assignment of a path $P$ such that all lists are of size  $4$.
Let $C_n$ be the number of non-repetitive colorings of the $n$ leftmost vertices of $P$ that respects $L$.
 Then for any integer $n< |P|$, we have
 $$C_{n+1}\ge2C_n$$
\end{lemma}
\begin{proof}
 Let us proceed by induction on $n$. 
 Let $n$ be an integer such that the Lemma holds for any integer smaller than $n$ and let us show that $C_{n+1}\ge2C_n$.
 Let $F$ be the set of colorings of the $n+1$ leftmost vertices that respect $L$ and are repetitive, but induce a non-repetitive coloring of the $n$ leftmost vertices. 
 Then
 \begin{equation}\label{coutingforpath}
  C_{n+1}=4C_n-|F|
 \end{equation}
Let us now bound the size of $F$. 
For every $i$, let $F_i$ be the colorings from $F$ that contain a square of length $2i$. Then clearly $|F|\le \sum_{i\ge1}|F_i|$.
For any coloring $c$ of $F_i$, the last $i$ colors of $c$ can be recovered from the previous $i$ colors (because of the repetition) and the first $n+1-i$ colors induce a non-repetitive coloring. We deduce that for all $i$,
$|F_i|\le C_{n+1-i}$. 
Now by induction hypothesis this implies that $|F_i|\le 2^{1-i}C_{n}$.
We finally get
\begin{equation}\label{conclusionlemmapath}
 C_{n+1}\ge4C_n-\sum_{i\ge1} 2^{1-i}C_{n}\ge 2C_n
\end{equation}
 which concludes our proof.
\end{proof}

The consequences of Lemma \ref{lemmpath} are in fact stronger than Theorem \ref{thmpath}, since it implies that there are at least $2^n$ coloring for any list assignment. However, the statement of this Lemma is even stronger and it is crucial in the proof that the number of allowed colorings is multiplied by at least $2$ every time that we color one more vertex.

\subsection{The proof technique}
Let us now give a more general informal sketch of our proof technique. 
Suppose you want to show that any graph from some class $\mathcal{C}$ of graphs admits a \emph{valid} coloring with at most $\gamma$ colors (eg. Theorem \ref{thmpath}).
Suppose moreover that both the class and the valid colorings are hereditary in the sense that the graph induced by a partial coloring also belongs to $\mathcal{C}$ and that every subcoloring of a valid coloring is also valid.
For any graph $G$, we let $c(G)$ be the set of valid colorings of $G$.

We try to show the stronger result that there exists a constant $\alpha$ such that for any graph $G\in\mathcal{C}$ and element $e$ of $G$ to be colored (this might be an edge or a vertex) we have the inequality
$|c(G)|\ge\alpha\times |c(G\setminus\{e\})|$ (eg, Lemma \ref{lemmpath} with $\gamma=4$ and $\alpha=2$).
We proceed by induction on the size of $G$ to show this result.

Let $F$ be the set of colorings of $G$ that are not valid but that induce
a valid coloring of $G\setminus\{e\}$. Then by definition we have the following equality $$|c(G)|= \gamma |c(G\setminus\{e\}| - F\,.$$ 
Now suppose we can find coefficients $(a_i)_{i\ge1}$ such that \begin{equation}\label{eqboundF}
|F|\le \sum_{i\ge1}a_i \frac{ |c(G\setminus\{e\}|}{\alpha^{i-1}}
\end{equation}
and
 \begin{equation}\label{eqgammaalpha}
\sum_{i\ge1}\frac{a_i}{\alpha^{i-1}}\le(\gamma-\alpha)\,.
 \end{equation}
We deduce
$$|c(G)|\ge |c(G\setminus\{e\})| \alpha$$ which conclude our proof.

The technical part is to show the upper-bound on $F$ with the right coefficients $(a_i)_{i\ge1}$. 
This is done by finding injections from $F$ to valid colorings of subgraphs of $G\setminus\{e\}$. More precisely, we find a way to express $F$ as the union of colorings $(F_i)_{i\ge1}$, 
such that for all $i$ there is an injection from $F_i$ to the union of the colorings of $a_i$ different subgraphs of $G\setminus\{e\}$ of cardinality $|G|-i$. Then we can use our induction hypothesis to upper-bound the number of valid colorings of any subgraph obtained by removing $i$ elements of $G\setminus\{e\}$ by $\frac{ |c(G\setminus\{e\}|}{\alpha^{i}}$ which leads to equation \eqref{eqboundF}.

It might seem that we have to guess the right values of $\gamma$ and $\alpha$ in the Lemma statement, but it is not the case. 
Indeed, one can first find the coefficients $(a_i)_{i\ge1}$ with variables $\gamma$ and $\alpha$ and then take the best values $\gamma$ and $\alpha$ such that equation \eqref{eqgammaalpha} is satisfied. This is done by choosing $\gamma= \min\limits_{\alpha>\frac{1}{r}}\left(\alpha+\sum\limits_{i\ge1}\frac{a_i}{\alpha^{i-1}}\right)$ where $r$ is the radius of convergence of $\sum_{i\ge1}a_i x^i$. \footnote{In pratice, we take upper-bounds  of the minimum obtained from manipulations of Taylor polynomials in order to give nicer expressions. These details only matter when writting the proof, but not when reading it.} 
For instance, in Lemma \ref{lemmpath}, we need $\gamma -\frac{\alpha}{\alpha-1}\ge \alpha$ and $\alpha>1$. The minimum of $\frac{\alpha}{\alpha-1}+\alpha$ is $4$ and is reached for $\alpha=2$ these are respectively the best values to take for $\gamma$ and $\alpha$.

The technique is strongly related to the entropy compression technique.
In fact, in the particular context of colorings of graphs of bounded degree, it is equivalent to the approach of 
\cite[Theorem 12]{MontassierEntropie}. 
Using our technique one can in fact provide a simpler proof of their Theorem 12 (our $a_i$ are their $C_i$ and our $\alpha$ is their $x^{-1}$), but it does not seem to be worth the trouble of  introducing all the necessary formalism only to provide an alternative proof of he exact same result. However, even if we can simplify the proof and match the bound of their Theorem, we cannot easily improve the bound.

Remark that bounding $|F|$ in the proof of Lemma \ref{lemmpath} is simplified by the linear structure of the path, since we know that the vertices that contribute to a square are always the last vertices added. 
In fact, in the setting of words this proof is almost identical to the power series method for pattern avoidance \cite{BELL20071295,BLANCHETSADRI201317,doublepat,rampersadpowerseries}.

\section{Non-repetitive colorings} 
In this section we apply our method to non-repetitive colorings of graphs of bounded maximum degree. 
\begin{theorem}\label{improvingthuenumber}
 For every graph $G$ with maximum degree $\Delta\ge1$, we have
 $$\pi_{\operatorname{ch}}(G)\le\Delta^2+\frac{3}{2^{\frac{2}{3}}}\Delta^{\frac{5}{3}}+2^{\frac{2}{3}}\Delta^{\frac{4}{3}}\,.$$
\end{theorem}

Let us instead show a stronger Lemma.
For any graph $G$ and any list assignment $L$ of $G$, the set $C_L(G)$ is the set of non-repetitive colorings of $G$ respecting the list assignment $L$.
\begin{lemma}
Let $\Delta\ge 2$ be an integer and $\gamma=\frac{3}{2^{\frac{2}{3}}}+2^{\frac{2}{3}}\Delta^{-\frac{1}{3}}+\Delta^{-\frac{2}{3}}$.
Let $G$ be a graph of maximal degree at most $\Delta$ and $L$ be a list assignment of $G$.
Suppose each list is of size at least $\Delta(\Delta-1)(1+ \gamma\Delta^{-\frac{1}{3}})+1$
then for any vertex $v$ of $G$ we have
 $$|C_L(G)|\ge \Delta(\Delta-1)(1+2^{\frac{1}{3}}\Delta^{-\frac{1}{3}}) |C_L(G\setminus\{v\})|\,.$$
\end{lemma}
\begin{proof}
 Let us show this by induction on the number of vertices of $G$.
 This is clearly true if $|G|=1$ since the empty graph has exactly one coloring.
 Let $n$ be an integer such that the Lemma holds for any graph with less than $n$ vertices.
 
 Let $G$ be a graph over $n$ vertices of maximal degree at most $\Delta$ and
 $L$ be a list assignment of $G$ such that each list is of size at least $\Delta(\Delta-1)(1+ \gamma\Delta^{-\frac{1}{3}})+1$.
 Let $v$ be any vertices of $G$.
 
 Let $F$ be the set of colorings of $G$  respecting $L$ that are repetitive and
 that induce a non-repetitive coloring of $G\setminus\{v\}$. We then have 
 \begin{equation}\label{maineqFnonrep}
  |C_L(G)|= (\Delta(\Delta-1)(1+ \gamma\Delta^{-\frac{1}{3}})+1)|C_L(G\setminus\{v\})|- |F|
 \end{equation}
We need to upper-bound the size of $F$.
Let $F_i$ be the set of colorings from $F$ that contain a path of length $2i$ that is a square.
Clearly $|F|\le\sum_{i\ge1} |F_i|$.
Thus for any coloring $c$ from $F_i$ there is path $p$ of length $2i$ such that 
\begin{itemize}
 \item $p$ induces a square in $c$,
 \item $p$ contains $v$ and we can call $p'$ the half of $p$ that contains $v$,
 \item the coloring induced over $G\setminus p'$ is non-repetitive,
 \item $p$ and the coloring induced over $G\setminus p'$ uniquely determines $c$ 
 (since the second half of the square is identical to the first half).
\end{itemize}
Thus, for any fixed $p$ and $p'$, the number of such colorings from $F_i$ is at most
$|C_L(G\setminus p')|$, but since $p'$ contains $v$ and $i-1$ other vertices our induction hypothesis implies that this quantity is bounded by
$$|C_L(G\setminus p')|\le \frac{|C_L(G\setminus \{v\})|}{(\Delta(\Delta-1)(1+2^{\frac{1}{3}}\Delta^{-\frac{1}{3}}))^{i-1}}\,.$$
Moreover, there are at most $i\Delta(\Delta-1)^{2i-2}$ paths of length $2i$ going through $v$.
To see that, remark that $v$ splits such a path in two halves, so one can choose the length of the shortest half between $0$ and $i-1$ and build the path by choosing all the vertices of the short half and then all the vertices of the long half in a DFS manner (there are $\Delta$ choices for the first vertex and $\Delta-1$ for each other vertex).
We deduce
$$|F_i|\le i\Delta(\Delta-1)^{2i-2} \frac{|C_L(G\setminus \{v\})|}{(\Delta(\Delta-1)(1+2^{\frac{1}{3}}\Delta^{-\frac{1}{3}}))^{i-1}}$$
If $i\ge2$ it implies, $$|F_i|\le i(\Delta-1) \frac{|C_L(G\setminus \{v\})|}{(1+2^{\frac{1}{3}}\Delta^{-\frac{1}{3}})^{i-1}}$$ 
and for $i=1$ we have 
$$|F_i|\le\Delta|C_L(G\setminus \{v\})| = (\Delta-1)|C_L(G\setminus \{v\})| +|C_L(G\setminus \{v\})|$$
Thus we can finally upper-bound $|F|$.
\begin{align*}
 |F|&\le|C_L(G\setminus \{v\})|\left( 1+(\Delta-1)\sum_{i\ge1}  \frac{i}{(1+2^{\frac{1}{3}}\Delta^{-\frac{1}{3}})^{i-1}}\right)\\
 |F|&\le|C_L(G\setminus \{v\})|\left(1+(\Delta-1)\frac{(1+2^{\frac{1}{3}}\Delta^{-\frac{1}{3}})^2}{((1+2^{\frac{1}{3}}\Delta^{-\frac{1}{3}})-1)^2}\right)\\
 |F|&\le|C_L(G\setminus \{v\})|\left(1+(\Delta-1)\Delta^{\frac{2}{3}}\left(2^{-\frac{1}{3}}+\Delta^{-\frac{1}{3}}\right)^2\right)
\end{align*}
Together with equation \eqref{maineqFnonrep}, it implies
\begin{align*}
|C_L(G)|&\ge|C_L(G\setminus\{v\})|\left( \Delta(\Delta-1)(1+\gamma\Delta^{-\frac{1}{3}})-(\Delta-1)\Delta^{\frac{2}{3}}\left(2^{-\frac{1}{3}}+\Delta^{-\frac{1}{3}}\right)^2\right)\\
|C_L(G)|&\ge |C_L(G\setminus\{v\})|\Delta(\Delta-1)\left(1+ \Delta^{-\frac{1}{3}}\left(\gamma-\left(2^{-\frac{1}{3}}+\Delta^{-\frac{1}{3}}\right)^2\right)\right)\\
|C_L(G)|&\ge |C_L(G\setminus\{v\})|\Delta(\Delta-1)\left(1+ \Delta^{-\frac{1}{3}}\left(\gamma-2^{-\frac{2}{3}}-2^{\frac{2}{3}}\Delta^{-\frac{1}{3}}-\Delta^{-\frac{2}{3}}\right)\right)
\end{align*}
Substituting $\gamma=\frac{3}{2^{\frac{2}{3}}}+2^{\frac{2}{3}}\Delta^{-\frac{1}{3}}+\Delta^{-\frac{2}{3}}$, we finally get
$$|C_L(G)|\ge |C_L(G\setminus\{v\})|\Delta(\Delta-1)\left(1+ \Delta^{-\frac{1}{3}}2^{\frac{1}{3}}\right)$$
which concludes this proof.
\end{proof}

Remark that the bound given by this Lemma is in fact
\begin{align*}
 \pi_{\operatorname{ch}}&\le\left\lceil\Delta(\Delta-1)\left(1+ \frac{3}{2^{\frac{2}{3}}}\Delta^{-\frac{1}{3}}+2^{\frac{2}{3}}\Delta^{-\frac{2}{3}}+\Delta^{-1}\right)+1\right\rceil\\
 \pi_{\operatorname{ch}}&\le\left\lceil\Delta^2+\frac{3}{2^{\frac{2}{3}}}\Delta^{\frac{5}{3}}+2^{\frac{2}{3}}\Delta^{\frac{4}{3}}-\frac{3}{2^{\frac{2}{3}}}\Delta^{\frac{2}{3}}-2^{\frac{2}{3}}\Delta^{\frac{1}{3}}\right\rceil
\end{align*}
 which is slightly stronger than the bound given in Theorem \ref{improvingthuenumber}.
One can also compare the bound from  Theorem \ref{improvingthuenumber} to the result from \cite{MontassierEntropie} (already mentioned in equation \eqref{ThueChoiceNumberToImprove}).
As $\Delta$ goes to $+\infty$ their upper-bound is asymptotically equivalent to
 $$\Delta^2+\frac{3}{2^\frac{2}{3}}\Delta^\frac{5}{3}+
 2^\frac{2}{3}\Delta^{\frac{4}{3}}+2\Delta + \mathcal{O}(\Delta^{\frac{2}{3}})
 $$ which is larger than our upper-bound by $2\Delta + \mathcal{O}(\Delta^{\frac{2}{3}})$. 
 This is a really minor improvement and this can be achieved with the entropy compression argument.
 The best known lowed-bound on the maximal Thue number for any maximum degree $\Delta$ is $\Omega\left(\frac{\Delta^2}{\log \Delta}\right)$ \cite{alongraph}, so this could be the case that even the first coefficient is not optimal.
 However, the fact that our method is simpler allowed us to easily improve the analysis while still providing a shorter proof.

\section{(Weak) total Thue coloring}
In this section, we need to consider three kinds of paths
\begin{itemize}
 \item \emph{vertex-paths}: sequences of consecutive adjacent vertices (they were simply called path in the previous section),
 \item \emph{edges-paths}: sequences of consecutive adjacent edges,
 \item \emph{mixed-paths}: alternating sequences of vertices and edges such that consecutive elements are adjacent.
\end{itemize}
In each of these definitions, we require that the paths are simple, that is each vertex or edge appears at most once in the path (we allow an edge path to go though the same vertex multiple time since it does not really matter). An \emph{element} of a graph is an edge or a vertex of the graph.

In this section, we need to color graphs element by element, but when we color an edge this might be the case that one or both of its vertices are not colored yet.
Thus the graph induced by the colored elements is not necessarily a proper graph in the sense that
some edges might be missing one or two vertices.
But for our inductive approach to hold, we need our result to hold for such objects.
We do not want to formalize this notion, but one way to properly dot it would be to define a graph as a pair of sets of objects (the vertices and the edges) and three relations (the adjacency relation between vertices, the adjacency relation between edges and the adjacency relation between vertices and edges).

Thus if $S$ is a set of edges or vertices of a graph $G$
then $G\setminus S$ is the graph obtained by deleting exactly the vertices and edges of $S$ (that is, we do not remove edges connected to some vertices of $S$ unless they also belong to $S$). 
Also if two edges are connected by a vertex $v$ then they are still considered to be connected in $G\setminus \{v\}$ and in particular they can still appear consecutively in an edge-path  of $G\setminus \{v\}$ (similarly two adjacent vertices are still adjacent even if we remove their shared edge and they can still be consecutive in a vertex path).
That is, the set of vertex paths (resp. edge paths) of $G\setminus S$
is the set of sequences of elements of $G\setminus S$ that are
vertex paths (resp. edge paths) of $G$.

\subsection{Weak total Thue coloring}\label{sectionwttc}
Given a graph $G$, a set $S$ of edges and vertices of $G$ and a list assignment $L$ of $G$, the set $C_L(G\setminus S)$ is the set of weak total Thue colorings of $G\setminus S$ respecting the list assignment of $L$ restricted to $G\setminus S$.
We are now ready to state our Theorem and the associated Lemma.
\begin{theorem}\label{weaktotalthuecoloringthm}
 For every graph $G$ with maximum degree $\Delta\ge1$, we have
 $$\pi_{\operatorname{T_wch}}(G)\le6\Delta\,.$$ 
\end{theorem}
\begin{lemma}\label{weaktotalthuelemma}
Let $\Delta\ge2$ be an integer.
Let $G$ be a graph of maximal degree less than $\Delta$ and $L$ be a list assignment of $G$.
Suppose each list is of size at least $6\Delta$
then for any vertex or edge $x$ of $G$:
 $$|C_L(G)|\ge 3\Delta|C_L(G\setminus\{x\})|\,.$$
\end{lemma}
\begin{proof}
 Let us show this by induction on the sum of the number of vertices and edges of $G$.
 This is true for the graph with a single vertex since the empty graph admits exactly one coloring.
 Let $n$ be an integer such that the Lemma holds for any graph with less than $n$ vertices and edges.
 
 Let $G=(V,E)$ be a graph with $|E|+|V|=n$ of maximal degree less than $\Delta$ and
 $L$ be a list assignment of $G$ such that each list is of size at least $6\Delta$.
 Let $x$ be an edge or a vertex of $G$.
 
 Let $F$ be the set of colorings of $G$ respecting $L$ that are weak total repetitive and
 that induce Weak total Thue coloring of $G\setminus\{x\}$. We then have
 \begin{equation}\label{maineqFweaknonrep}
  |C_L(G)|= 6\Delta|C_L(G\setminus\{x\})|- |F|
 \end{equation}
We need to upper-bound the size of $F$.
Let $F_i$ be the set of colorings from $F$ that contains a mixed-path of length $2i$ that induces a square.
Clearly $|F|\le\sum_{i\ge1} |F_i|$.
Thus for any coloring $c$ from $F_i$ there is mixed-path $p$ of length $2i$ such that 
\begin{itemize}
 \item $p$ induces a square in $c$,
 \item $p$ contains $x$ and we can call $p'$ the half of $p$ that contains $x$,
 \item the coloring induced over $G\setminus p'$ is non-repetitive,
 \item $p$ and the coloring induced over $G\setminus p'$ uniquely determines $c$ 
 (since the second half of the square is identical to the first half).
\end{itemize}
Given $p$ and $p'$ the number of such coloring from $F_i$ is at most
$|C_L(G\setminus p')|$, but since $p'$ contains $x$ and $n-1$ other elements our induction hypothesis implies that this quantity is bounded by
$$|C_L(G\setminus p')|\le \frac{|C_L(G\setminus \{x\})|}{(3\Delta)^{i-1}}\,.$$
If $x$ is a vertex then there are at most $i\Delta^i$ mixed-paths of length $2i$ going through $x$.
If $x$ is an edge then there are at most $2i\Delta^{i-1}\le i\Delta^i$ mixed-paths of length $2i$ going through $x$.
We deduce
$$|F_i|\le i\Delta^{i} \frac{|C_L(G\setminus \{x\})|}{(3\Delta)^{i-1}}= i\Delta \frac{|C_L(G\setminus \{x\})|}{3^{i-1}}$$
Thus we can finally upper-bound $|F|$
\begin{align*}
 |F|&\le\Delta |C_L(G\setminus \{x\})|\sum_{i\ge1}  \frac{i}{3^{i-1}}\\
 |F|&\le\frac{9}{4}\Delta |C_L(G\setminus \{x\})|
\end{align*}
Together with equation \eqref{maineqFweaknonrep}, it implies
\begin{align*}
|C_L(G)|&\ge 6\Delta|C_L(G\setminus\{x\})|-\frac{9}{4}\Delta |C_L(G\setminus \{x\})|\\
|C_L(G)|&\ge \frac{15}{4}\Delta|C_L(G\setminus\{x\})|\ge3\Delta|C_L(G\setminus\{x\})|
\end{align*}
which concludes our proof.
\end{proof}
One slightly improve the leading coefficient by making the last inequality of the proof tight ($15/4\ge3$ is not really tight) and a better analysis leads to a coefficient $\gamma=5.21914$ instead of $6$ ($\gamma$ is a root of the polynomial $-3 - 4 x - 20 x^2 + 4 x^3$). 
The same result can be showed with a longer proof relying on entropy-compression.
However, in the next subsection we slightly improve this result for large values of $\Delta$ (replacing $6$ by $4.25$) and this improved bound does not seems easy to reproduce with entropy compression or LLL.

\subsection{Weak total Thue coloring for large maximal degree}
We improve our bound on the weak total Thue number of graphs for large values of $\Delta$.
\begin{theorem}\label{weaktotalthuecoloringthmimproved}
 For every graph $G$ with maximum degree $\Delta\ge 300$, we have
 $$\pi_{\operatorname{T_wch}}(G)\le\lceil4.25\Delta\rceil\,.$$ 
\end{theorem}
This is a simple corollary of the following Lemma.
\begin{lemma}
Let $\Delta\ge300$ be an integer.
Let $G$ be a graph of maximal degree less than $\Delta$ and $L$ be a list assignment of $G$.
Suppose each list is of size at least $4.25\Delta$
then for any vertex $v$ of $G$:
 $$|C_L(G)|\ge 1.62\Delta|C_L(G\setminus\{v\})|\,.$$
  and for any edge $e$ of $G$:
 $$|C_L(G)|\ge 4.2\Delta|C_L(G\setminus\{e\})|\,.$$
\end{lemma}
\begin{proof}
The proof is rather similar to the proof of Lemma \ref{weaktotalthuelemma} We proceed with the same induction with the following difference, but we need to distinguish between edges and vertices.

Let us start with the case where $e$ is an edge.
We can use the exact same argument as in the previous proof with the fact that there are at most $2i\Delta^{i-1}$ mixed-paths of length $2i$ going through $e$ and the fact that for any element $x$ of $G$
$|C_L(G)|\ge 1.62\Delta|C_L(G\setminus\{x\})|\,.$ to deduce
\begin{align*}
|C_L(G)|&\ge 4.25\Delta|C_L(G\setminus\{e\})|- \sum_{i\ge1}\frac{2i\Delta^{i-1}|C_L(G\setminus\{e\})|}{(1.62\Delta)^{i-1}}\\
|C_L(G)|&\ge|C_L(G\setminus\{e\})|\Delta \left(4.25- \frac{2}{\Delta}\sum_{i\ge1}\frac{i}{1.62^{i-1}}\right)\\
|C_L(G)|&\ge|C_L(G\setminus\{e\})|\Delta \left(4.25- \frac{2}{\Delta}\frac{1.62^{2}}{0.62^2}\right)
\end{align*}
Since $\Delta>300$, numerical computations give
$$|C_L(G)|\ge 4.2\Delta|C_L(G\setminus\{e\})|\,.$$ which conclude this case.

If $v$ is a vertex then the half $p'$ of a mixed path of length $2i$ containing $x$ contains at least $\lceil \frac{i-1}{2}\rceil$ edges. 
Thus in this case our induction hypothesis implies
$$|C_L(G\setminus p')|\le \frac{|C_L(G\setminus \{x\})|}{(\sqrt{1.62\times4.2}\Delta)^{i-1}}\,.$$
Moreover, there are at most $i\Delta^{i}$ mixed-paths of length $2i$ going through $v$ which gives:
\begin{align*}
|C_L(G)|&\ge 4.25\Delta|C_L(G\setminus\{e\})|- \sum_{i\ge1}\frac{i\Delta^{i}|C_L(G\setminus\{e\})|}{(\sqrt{1.62\times4.2}\Delta)^{i-1}}\\
|C_L(G)|&\ge|C_L(G\setminus\{e\})|\Delta \left(4.25-\sum_{i\ge1}\frac{i}{(\sqrt{1.62\times4.2})^{i-1}}\right)
\end{align*}
Numerical computations give 
 $$|C_L(G)|\ge 1.62\Delta|C_L(G\setminus\{v\})|$$
which conclude our proof.
\end{proof}

By exploiting the fact that edges behave in a slightly better way than vertices and by using two different ``growth rate'' to distinguish between these cases, we were able to show a slightly stronger result. 
It is not clear whether this proof can be adapted to the entropy compression method and if so it probably requires a different approach or a really complicated analysis.

\subsection{Total Thue coloring}\label{sectionttc}
For any set $S$ of edges and vertices of a graph $G$ and any list assignment of $G$, the set $C_L(G\setminus S)$ is the set of total Thue colorings of $G\setminus S$ respecting the list assignment of $L$ restricted to $G\setminus S$. 

\begin{theorem}\label{totalthuecoloringthm}
 For every graph $G$ with maximum degree $\Delta$, we have
 $$\pi_{\operatorname{Tch}}(G)\le
 \Delta^2+\frac{3}{2^\frac{1}{3}}\Delta^{\frac{5}{3}}+8\Delta^{\frac{4}{3}}+1\,.$$ 
\end{theorem}
This is a simple Corollary of the following Lemma.
\begin{lemma}
Let $\Delta\ge2$ be an integer and $\gamma =\frac{3}{2^{\frac{1}{3}}}+8\Delta^{-\frac{1}{3}}$.
Let $G$ be a graph of maximal degree less than $\Delta$ and $L$ be a list assignment of $G$.
Suppose each list is of size at least $\Delta^2(1+\gamma\Delta^{-\frac{1}{3}})$
then for any vertex or edge $x$ of $G$:
 $$|C_L(G)|\ge \Delta^2\left(1+2^{\frac{2}{3}}\Delta^{-\frac{1}{3}}\right)|C_L(G\setminus\{x\})|\,.$$
\end{lemma}\begin{proof}
 Let us show this by induction on the sum of the number of vertices and edges of $G$.
 This is true for the graph with a single vertex since the empty graph admits exactly one coloring.
 Let $n$ be an integer such that the Lemma holds for any graph with less than $n$ vertices and edges.
 
 Let $G=(V,E)$ be a graph with $|E|+|V|=n$ of maximal degree less than $\Delta$ and
 $L$ be a list assignment of $G$ such that each list is of size at least $\Delta^2(1+\gamma\Delta^{-\frac{1}{3}})$.
 Let $x$ be an edge or a vertex of $G$.
 
 Let $F$ be the set of colorings of $G$  respecting $L$ that are not-total Thue coloring and
 that induce a total Thue coloring of $G\setminus\{x\}$. We then have
 \begin{equation}\label{maineqtotalnonrep}
  |C_L(G)|= \Delta^2(1+\gamma\Delta^{-\frac{1}{3}}) |C_L(G\setminus\{x\})|- |F|
 \end{equation}
We need to upper-bound the size of $F$.
Let $M_i$ be the set of colorings from $F$ that contains a mixed-path of length $2i$ inducing a square.
If $x$ is an edge (resp. a vertex) let $S_i$ be the set 
of colorings from $F$ that contain an edge-path (resp. a vertex-path) of length $2i$ inducing a square.
Clearly $|F|\le\sum_{i\ge1} (|M_i|+|S_i|)$.

For any coloring $c$ from $M_i$ there is mixed-path $p$ of length $2i$ such that 
\begin{itemize}
 \item $p$ induces a square in $c$,
 \item $p$ contains $x$ and we can call $p'$ the half of $p$ that contains $x$,
 \item the coloring induced over $G\setminus p'$ is non-repetitive,
 \item $p$ and the coloring induced over $G\setminus p'$ uniquely determines $c$ 
 (since the second half of the square is identical to the first half).
\end{itemize}

Given $p$ and $p'$ the number of such coloring from $M_i$ is at most
$|C_L(G\setminus p')|$, but since $p'$ contains $x$ and $n-1$ other elements our induction hypothesis implies that this quantity is bounded by
$$|C_L(G\setminus p')|\le \frac{|C_L(G\setminus \{x\})|}{(\Delta^2(1+2^{\frac{2}{3}}\Delta^{-\frac{1}{3}}))^{i-1}}\,.$$
If $x$ is a vertex then there are at most $i\Delta^i$ mixed-paths of length $2i$ going through $x$.
If $x$ is an edge then there are at most $2i\Delta^{i-1}\le i\Delta^i$ mixed-paths of length $2i$ going through $x$.
We deduce
$$|M_i|\le i\Delta^{i} \frac{|C_L(G\setminus \{x\})|}{(\Delta^2(1+2^{\frac{2}{3}}\Delta^{-\frac{1}{3}}))^{i-1}}= i\Delta \frac{|C_L(G\setminus \{x\})|}{(\Delta(1+2^{\frac{2}{3}}\Delta^{-\frac{1}{3}}))^{i-1}}$$

If an edge-path (resp. a vertex path) induces a square we can also recover the coloring of the full path
by knowing only the first half of it. Moreover, there are at most 
$2i\Delta^{2i-1}$ edge-paths of length $2i$ going through a given edge and at most
$i\Delta^{2i-1}\le2i\Delta^{2i-1}$ vertex-paths of length $2i$ going through a given vertex.
Thus following the same idea we can bound the size of $S_i$ by:

$$|S_i|\le 2i\Delta^{2i-1} \frac{|C_L(G\setminus \{x\})|}{(\Delta^2(1+2^{\frac{2}{3}}\Delta^{-\frac{1}{3}}))^{i-1}}= 2i\Delta \frac{|C_L(G\setminus \{x\})|}{(1+2^{\frac{2}{3}}\Delta^{-\frac{1}{3}})^{i-1}}$$
Thus we can upper-bound $|F|$
\begin{align*}
 |F|&\le\Delta |C_L(G\setminus \{x\})|\sum_{i\ge1}  \frac{2i}{(1+2^{\frac{2}{3}}\Delta^{-\frac{1}{3}})^{i-1}}
 + \frac{i}{(\Delta(1+2^{\frac{2}{3}}\Delta^{-\frac{1}{3}}))^{i-1}}\\
 |F|&\le\Delta |C_L(G\setminus \{x\})|\left(1+\left(2+\frac{1}{\Delta}\right)\sum_{i\ge1}  \frac{i}{(1+2^{\frac{2}{3}}\Delta^{-\frac{1}{3}})^{i-1}}\right)\\
 |F|&\le\Delta |C_L(G\setminus \{x\})|\left(1+\left(2+\frac{1}{\Delta}\right)\frac{(1+2^{\frac{2}{3}}\Delta^{-\frac{1}{3}})^2}{(2^{\frac{2}{3}}\Delta^{-\frac{1}{3}})^2}\right)\\
 |F|&\le\Delta^{\frac{5}{3}} |C_L(G\setminus \{x\})| \left(\Delta^{-\frac{2}{3}}+\left(2+\frac{1}{\Delta}\right)
\left(2^{-\frac{2}{3}}+\Delta^{-\frac{1}{3}}\right)^2\right)
 \end{align*}
 Since $\Delta>1$ (and $\Delta ^{-\frac{1}{3}}> \Delta ^{-\frac{2}{3}}$) we finally get
$$
 |F|\le\Delta^{\frac{5}{3}} |C_L(G\setminus \{x\})| \left(
 2^{-\frac{1}{3}}+8\Delta^{-\frac{1}{3}}
 \right)\,.$$
Together with equation \eqref{maineqtotalnonrep}, it implies
\begin{align*}
|C_L(G)|&\ge \Delta^2(1+\gamma \Delta^{-\frac{1}{3}})|C_L(G\setminus\{x\})|-\Delta^{\frac{5}{3}} |C_L(G\setminus \{x\})| \left(
 2^{-\frac{1}{3}}+8\Delta^{-\frac{1}{3}}
 \right)\\
|C_L(G)|&\ge \Delta^2|C_L(G\setminus\{x\})|\left(1+\Delta^{-\frac{1}{3}}\left(\gamma-\left(
 2^{-\frac{1}{3}}+8\Delta^{-\frac{1}{3}}
 \right)\right)\right)
\end{align*}
Substituting $\gamma =\frac{3}{2^{\frac{1}{3}}}+8\Delta^{-\frac{1}{3}}$, we finally get
$$|C_L(G)|\ge  \Delta^2\left(1+2^{\frac{2}{3}}\Delta^{-\frac{1}{3}}\right)|C_L(G\setminus\{v\})|$$
which concludes this proof.
\end{proof}
This upper-bound also holds for Thue coloring and edge Thue-coloring which provides the following Corollary.
\begin{corollary}
 For every graph $G$ with maximum degree $\Delta$, we have
 $$\pi'_{\operatorname{ch}}(G)\le
 \Delta^2+\frac{3}{2^\frac{1}{3}}\Delta^{\frac{5}{3}}+8\Delta^{\frac{4}{3}}+1\,.$$ 
\end{corollary}
This bound is better than the upper-bound
$\pi'_l(G)\le \Delta^2+2^{\frac{4}{3}}\Delta^{\frac{5}{3}}+O(\Delta^{\frac{4}{3}})$ given in 
\cite{MontassierEntropie} ($2^{\frac{4}{3}}\approx2.52...$ and $\frac{3}{2^\frac{1}{3}}\approx2.38...$).
Once again a more detailed analysis in their argument certainly provides the same bound.

Remark that one easily improves the coefficient of $\Delta^\frac{4}{3}$  with a more detailed analysis (at least as low as $1+2^{4/3}$ in the case of total coloring and $2^{4/3}$ in the case of edge coloring).

\section{Conclusion}
As already stated multiple times, most result in this paper can be obtained with entropy-compression arguments and it seems to be the case that these two approaches are in fact equivalent. 
However, our approach is much simpler to use and in particular it is not clear how to adapt the proof of Theorem \ref{weaktotalthuecoloringthmimproved} to the entropy-compression method.
The approach can obviously be generalized outside of the scope of non-repetitive coloring and we can for instance provide simpler proofs of all the result from \cite{ESPERET20131019}.

Our proof technique is also strongly related to the Lov\'asz Local Lemma.
We show lemmas of the form ``with $\gamma$ colors, coloring a new vertex multiply the number of valid coloring by $\alpha$'' and the second part of this statement can be replaced by ``when adding a new vertex the probability for a random coloring to be valid is at least multiplied by $\frac{\alpha}{\gamma}$'' which is the idea behind LLL (remark that $\frac{\alpha}{\gamma}$ is a quantity smaller than $1$). 
Following the same idea, this is in fact possible to show results equivalent to the SAT versions of LLL, by simply rewriting the standard inductive proof without using probability. However, such proof is really not informative since it simply follows the proof of LLL.

In this application of our approach that is similar to LLL, we need to show that the number of solutions does not decrease too fast every time that we add a constraint. In the applications to graph colorings, we instead add the colored elements one by one and we show that at each step the number of solutions grows fast enough. That is, the main difference is that in the first case we add the constraints one by one and on the second case it is better to add them  several at a time using the underlying structure of the problem.  

\section*{Acknowledgement}
I wish to thank Gwena\"el Joret and William Lochet for comments on earlier drafts. 
I also wish to thank Lucile without whom the covid-19 lockdown would have been much more unpleasant and less productive.

\bibliographystyle{plain}
\bibliography{biblio}
\end{document}